\documentclass[11pt]{amsart} \usepackage{amsfonts}

\newtheorem{theorem}{Theorem}[section]
\newtheorem{lemma}[theorem]{Lemma}
\newtheorem{proposition}[theorem]{Proposition}
\newtheorem{corollary}[theorem]{Corollary}

\DeclareMathSymbol{\Z}{\mathbin}{AMSb}{"5A}
\DeclareMathSymbol{\C}{\mathbin}{AMSb}{"43}
\DeclareMathSymbol{\N}{\mathbin}{AMSb}{"4E}
\DeclareMathSymbol{\Q}{\mathbin}{AMSb}{"51}
\DeclareMathSymbol{\R}{\mathbin}{AMSb}{"52}

\begin{document}

\title[On the Ramsey numbers of trees with small diameter]{On the
Ramsey numbers of trees with small diameter}

\author[Patrick Bahls]{Patrick Bahls}

\email{pbahls@unca.edu}

\author[T. Scott Spencer]{T. Scott Spencer}

\email{tsspence@unca.edu}



\begin{abstract} We estimate the Ramsey number $r(T)=r(T,T)$ for various trees $T$, obtaining a precise value for $r(T)$ for a large number of trees of diameter $3$.  Furthermore we prove that all trees of diameter $3$ are Ramsey unsaturated as defined by Balister, Lehel, and Schelp in their article ``Ramsey unsaturated and saturated graphs.'' \end{abstract}

\maketitle

\section{Introduction} \label{introduction}

\noindent Our goal in this paper is to establish some new results concerning Ramsey numbers of various types of trees (connected acyclic graphs).  In particular, we will establish upper and lower bounds on $r(T,T)$ for various caterpillars (obtaining precise values for certain trees $T$) and prove that some of these caterpillars, which we call bistars, are Ramsey unsaturated in the sense of \cite{balisterlehelschelp} and \cite{schelp}.  Though it is conjectured (see \cite{balisterlehelschelp} and \cite{schelp}) that ``most'' graphs are unsaturated, the latter result extends the rather short list of graphs explicitly known to be unsaturated.

We begin by recalling some basic definitions regarding Ramsey numbers.  Let $G_1=(V_1,E_1)$ and $G_2=(V_2,E_2)$ be graphs.  For our purposes, the graphs $G_i$ will always be simple, undirected, and connected.  The \textit{Ramsey number of} $G_1$ \textit{versus} $G_2$, denoted $r(G_1,G_2)$ is the minimum value $r$ such that any $2$-coloring of the edges of the complete graph $K_r$ yields a monochromatic subgraph isomorphic to one of $G_1$ or $G_2$.  The number $r(G_1,G_1)$ will be abbreviated to $r(G_1)$ and we will call this simply the \textit{Ramsey number of} $G_1$.  Though it is easy to extend the definition to consider an arbitrary number of $k$ colors, yielding the obvious definition for $r(G_1,...,G_k)$, we will not be concerned with this generalization here.  (See \cite{grahamrothschildspencer} for a thorough reference.)

We note that Ramsey numbers of trees versus trees have not been very well investigated.  In \cite{burrroberts}, Burr and Roberts establish the Ramsey number of any number of star versus one another.  A year later Burr (in \cite{burr}) proved that for any tree with $m$ vertices and any star with $n$ leaves satisfying $(m-1)|(n-1)$, $r(T,S_n)=m+n-1$.  In \cite{guovolkmann}, Guo and Volkmann establish upper bounds on the Ramsey number of tree/generalized-claw pairs, and very limited classes of bistars.  In \cite{erdosgraham}, Erd\H{o}s and Graham prove some very general but relatively weak bounds on $r(T)$ for an arbitrary tree $T$.  We will note in the following section how our results improve upon theirs in the cases of interest to us.

Specifically, we concern ourselves here with caterpillars.  A \textit{caterpillar} $C$ is any tree such that the removal of all of $C$'s leaves results either in a path of nonnegative length or in the empty set.  For a caterpillar $C=(V,E)$, the set $S(C) = \{ v \in V \ | \ d(v)>1 \}$, where $d(v)$ is the degree of $v$, will be called the \textit{spine} of $C$.

We wish to estimate $r(C)$ for any caterpillar $C$, obtaining precise values for this number whenever possible.  Our analysis will be most complete when $S(C)$ is small.  For instance, if $C$ is a star (that is, if $|S(C)|=1$), it is not hard to show that $r(C) \leq 2n$ (see Proposition~\ref{star}).  In the case of \textit{bistars} ($|S(C)|=2$), we obtain precise values for $r(C)$ in certain cases and fairly strong upper and lower bounds in others (see Propositions~\ref{propLowerbound}, \ref{upperboundbign}, and \ref{upperboundsmalln}; note that bistars are precisely the trees with diameter $3$).  For still larger $S(C)$ precise values are harder to come by.

Having established some bounds for $r(C)$ we will turn to the issue of saturation.  We say that a graph $G=(V,E)$, $G$ not complete, is \textit{Ramsey saturated} (or merely \textit{saturated}) if given any $e \in (V \times V) \setminus E$, the graph $G+e = (V,E \cup \{e\})$ satisfies $r(G+e)>r(G)$.  That is, a graph is saturated if any supergraph formed by adding an edge not already present results in a graph with a strictly greater Ramsey number.  We say that $G$ is \textit{unsaturated} otherwise.  One of our primary results (Theorem~\ref{bistarunsat}) states that bistars are unsaturated.  This result is unsurprising; indeed, in \cite{schelp} Schelp conjectures that every non-star tree $T$ satisfying $|V(T)| \geq 5$ is unsaturated, though proving every member of a specific family of trees is Ramsey unsaturated is typically very difficult.

Let us now begin by establishing some results concerning general caterpillars.

\section{A lower bound on $r(C)$ for a caterpillar $C$} \label{secLowerbound}

\noindent The following notation will be convenient.  As we will consider only two colors, we will consistently refer to edges as either ``blue'' edges or ``red'' edges.  For a given vertex $v$, we let \[ B_v = \{u \in V \ | \ \{u,v\} \in E, u \ {\rm blue} \} \] \noindent and we define the \textit{blue degree} of $v$ by $d_b(v) = |B_v|$.  We define $R_v$ and the \textit{red degree}, $d_r(v)=|R_v|$, similarly.  When dealing with vertices from an indexed set, $v=v_i$, we denote $B_{v_i}$ and $R_{v_i}$ by $B_i$ and $R_i$, respectively.

The simplest example of a caterpillar is the star $S_n=K_{1,n}$, consisting of a single vertex of degree $n$, each of its neighbors a leaf.  The following result is easy to prove.  We offer a prove here to give a flavor of our later arguments; see also \cite{burr}, \cite{burrroberts}, or \cite{grahamrothschildspencer}.

\begin{proposition} \label{star}
Let $n \in \mathbb{N}$.  Then \[ r(S_n) = \left\{ \begin{array}{rc} 2n-1 & n \ {\rm even,} \\ 2n & n \ {\rm odd.} \\ \end{array} \right. \]
\end{proposition}

\begin{proof}
Consider a $2$-coloring of the graph $K_{2n}$.  By pigeonhole principle, every vertex is incident at least $n$ vertices of one color or the other, giving the required monochromatic star.  Thus $r(S_n) \leq 2n$.

Suppose that $n$ is odd, and let $\{v_1,...,v_{2n-1}\}$ be the vertices of $K_{2n-1}$.  Color every edge $\{v_i,v_{i+j}\}$ blue, for $j \in \{1,...,\frac{n-1}{2}\}$, where addition is performed modulo $2n-1$.  Color every remaining edge red.  The resulting $2$-coloring satisfies $d_b(v_i)=d_r(v_i)=n-1$ for every $i$, so that there is no monochromatic $S_n$.  Thus $r(S_n)=2n$ if $n$ is odd.

Now suppose that $n$ is even.  In order to $2$-color $K_{2n-1}$ without obtaining a monochromatic $S_n$ we must ensure that $d_b(v)=d_r(v)=n-1$ for every vertex $v$ in $K_{2n-1}$.  This coloring would yield the odd number $\frac{(2n-1)(n-1)}{2}$ of edges of each color, a clear contradiction.  Thus $r(S_n) \leq 2n-1$.

Finally, consider two copies of $K_{n-1}$, every edge of each of which is colored blue.  Connect each vertex of one with each vertex of the other by a red edge.  The result is a $2$-coloring of the graph $K_{2n-2}$ without a monochromatic $S_n$.  Therefore $r(S_n)>2n-2$, proving that $r(S_n)=2n-1$ after all. \end{proof}

Our next result establishes a lower bound on the Ramsey number $r(C)$ for an arbitrary caterpillar $C=(V,E)$ with at least $2$ vertices on its spine.  Let $C(n_1,...,n_k)$ denote the caterpillar with spine $\{v_1,...,v_k\}$ in which $n_i$ leaves are adjacent to $v_i$.  Without loss of generality we may assume that $n_1 \neq 0$ and $n_k \neq 0$, as otherwise we would be able to decrease the length of the spine.

\begin{proposition} \label{propLowerbound} Let $k \geq 2$ and let $n_i \in \mathbb{N}_0$ such that $n_1 \neq 0$ and $n_k \neq 0$.  Let \[ m_1 = \sum_{i=1}^{\lceil k/2 \rceil} n_{2i-1} + \lfloor k/2 \rfloor \hskip 5mm {\rm and} \hskip 5mm m_2 = \sum_{i=1}^{\lfloor k/2 \rfloor} n_{2i} + \lceil k/2 \rceil. \] \noindent Finally, let $m=\min\{m_1,m_2\}$.  Then \[ r(C(n_1,...,n_k)) \geq \sum_{i=1}^k n_i + k + m - 1 = |V|+m-1. \] \end{proposition}

\begin{proof} Let $K_{|V|-1}$ and $K_{m-1}$ ($m$ as in the statement of the proposition) be complete graphs, every edge of both of which is colored blue.  Connect each vertex of one of these graphs with each vertex of the other with a red edge.  Since $C=C(n_1,...,n_k)$ is connected there are insufficiently many vertices in either $K_{|V|-1}$ or $K_{m-1}$ for us to find a blue copy of $C$.  Moreover, in any red copy of $C$ the odd-indexed spine vertices would lie in one of the blue complete graphs, and the even-indexed spine vertices would lie in the other.  The number $m$ has been chosen so that regardless of which of these two choices is made there are insufficiently many vertices in $K_{m-1}$ to account for all of the vertices of $C$ which would have to lie in $K_{m-1}$.

For example, suppose $v_1,v_3,...$ lie in $K_{|V|-1}$.  Then the $n_1+n_3+ \cdots$ leaves adjacent to these vertices, as well as the spine vertices $v_2,v_4,...$, would have to lie in $K_{m-1}$, $m_1$ vertices in all.  Since $m_1>m-1$, this cannot be.  A similar contradiction arises should $v_2,v_4,...$ lie in $K_{|V|-1}$ instead.

There is therefore no monochromatic subgraph of the form $C$ in this $2$-coloring of $K_{|V|+m-2}$. \end{proof}

\noindent \textbf{Examples.} \ We consider a few special cases.

\begin{enumerate}

\item \textbf{Paths.} \ Let $P_k$ be the path on $k$ vertices, $k \geq 4$.  The reader may check that our formula yields the same bound whether we consider $P_k$ as a caterpillar with $k$, $k-1$, or $k-2$ vertices on its spine.  For simplicity, we let $S(P_k)=P_k$, so that $n_i=0$ for all $i$, and $m=\min\{\lfloor k/2 \rfloor,\lceil k/2 \rceil\} = \lfloor k/2 \rfloor$, so that $r(P_k) \geq k+\lfloor k/2 \rfloor - 1$.

\item \textbf{Bistars.} \ Let $B(n_1,n_2)$ be the bistar with $2$ spine vertices, and without loss of generality let $n_1 \leq n_2$.  Here $k=2$ and $m=n_1+1$, so that $r(B(n_1,n_2)) \geq 2n_1+n_2+2$.

\item \textbf{Regular caterpillars.} \ If $k$ is arbitrary but $n_i=n$ for all $i$, Proposition~\ref{propLowerbound} yields \[ r(C(n,...,n)) = \left\{ \begin{array}{rl} \frac{3k(n+1)-2}{2} & k \ {\rm even}, \\ \frac{(3k-1)(n+1)}{2} & k \ {\rm odd}. \end{array} \right. \]

\end{enumerate}

The last example makes it clear that our lower bound grows linearly in both $n$ and $k$.  Moreover, the bound we obtain for bistars is dramatically better than those obtained by applying a general result due to Erd\H{o}s and Graham \cite{erdosgraham}, which gives us \[ r(B(n_1,n_2)) > (|E(B(n_1,n_2))|-1)\lfloor\frac{3}{2}\rfloor = n_1+n_2-2. \] \noindent Results from \cite{erdosgraham} also give an upper bound on $r(C)$ for a caterpillar; namely, \[ r(C) \leq 4|E(C)|+1 = 4|V(C)|-3. \] \noindent The upper bounds we will obtain in the following section for certain bistars are substantially stronger than this.

\section{The Ramsey numbers of bistars} \label{bistars}

\noindent We now focus our attention on bistars, trees with diameter $3$.  Let $B(m,n)$ denote the bistar with spine vertices (called \textit{centers}) of degree $m+1$ and $n+1$, and throughout this section we assume that $m \leq n$.  With this notation, Proposition~\ref{propLowerbound} gives $r(B(m,n)) \geq 2m+n+2$.

For small values of $m$, $B(m,n)$ is very similar to the star $S_{n+1}$, so the following result is not very surprising:

\begin{proposition} \label{smallm} Let $m \in \{1,2\}$.
\begin{enumerate}
\item $r(B(1,1)) = r(P_4) = 5$ and $r(B(1,n)) = r(S_{n+1}) = \left\{ \begin{array}{rl} 2n+1 & n \ {\rm odd,} \\ 2n+2 & n \ {\rm even}. \\ \end{array} \right.$ for all $n \geq 2$.

\item $r(B(2,2)) = 8$ and $r(B(2,n)) \leq 2n+3$ for all $n \geq 3$.
\end{enumerate}
\end{proposition}

\begin{proof}
To prove (1) we begin by noting that $B(1,1) = P_4$ and $r(P_4) \geq 5$ by Proposition~\ref{propLowerbound}, and the reverse inequality is not hard to show using the pigeonhole principle.  For $n \geq 2$, note that since $B(1,n)$ contains $S_{n+1}$ as a subgraph, $r(B(1,n)) \geq r(S_{n+1})$.

We prove the reverse inequality when $n$ is odd; the case for even $n$ is analogous.  If $n$ is odd, $n+1$ is even, so that $r(S_{n+1}) = 2(n+1)-1 = 2n+1$.  Consider a $2$-colored copy of $K_{2n+1}$ and without loss of generality suppose the $v=V(K_{2n+1})$ is the center of a blue star $S_{n+1}$ with leaves $\{v_1,...,v_{n+1}\}$.  In order to avoid a blue copy of $B(1,n)$ every edge $\{v_i,u_j\}$ (for $u_j \not\in \{v,v_1,...,v_{n+1}\}$) must be red.  However, then we obtain a red copy of $B(1,n)$ with centers $v_i$ and $u_j$ for any $i$ and $j$.  Therefore $r(B(1,n)) \leq 2n+1$, and equality must hold.

To prove (2), first consider $B(2,2)$.  Proposition~\ref{propLowerbound} shows that $r(B(2,2)) \geq 8$.  To obtain the reverse inequality, we will require the following lemma, which will be of crucial importance in establishing many of our later results as well:

\begin{lemma} \label{maxcoldegree} Let $m,n \in \mathbb{N}$ and let $R \geq 2m+n+2$.  Consider a $2$-coloring of the edges of $K_R$.  If there exists a vertex $v \in V = V(K_R)$ such that $\max\{d_b(v),d_r(v)\} \geq m+n+1$, then $K_R$ contains a monochromatic $B(m,n)$. \end{lemma}

\begin{proof} Suppose, without loss of generality, that $d_b(v_1) \geq m+n+1$, and consider any $v_2 \in B_1$.

First suppose $d_b(v_2) \geq m+1$.  In this case it is easy to see that $v_1$ and $v_2$ form the centers of a blue bistar $B(m,n)$, regardless of the size of $B_1 \cap B_2$.

Thus we may assume $d_b(v_2) \leq m$, so that $d_r(v_2) \geq R-m-1 \geq m+n+1$.  By an argument like that in the previous paragraph we see that $v \in R_2 \Rightarrow d_r(v) \leq m$.  Therefore if $v \in B_1 \cap R_2$, we know that $d(v) = d_b(v)+d_r(v) \leq 2m$, a contradiction, since $d(v) \geq 2m+n+1 > 2m$.

This forces $B_1 \cap R_2 = \emptyset$.  However, this too is a contradiction, since $d_b(v_1) \geq m+n+1$ and $d_r(v_2) \geq m+n+1$, so the pigeonhole principle implies there must be nontrivial intersection between $B_1$ and $R_2$.

We are now done, as every possibility yields either a monochromatic copy of $B(m,n)$ or a contradiction. \end{proof}

A slight restatement of the above lemma gives us bounds on both $d_b(v)$ and $d_r(v)$ for $v \in K_R$ as above:

\begin{corollary} \label{boundsondegrees} Let $m,n \in \mathbb{N}$ such that $m \leq n$ and $R \geq 2m+n+2$, and consider a $2$-coloring of the edges of $K_R$.  If $K_R$ contains no monochromatic bistar $B(m,n)$, then $R-m-n-1 \leq d_b(v),d_r(v) \leq m+n$ for all $v \in V(K_R)$. \end{corollary}

We now return to proving that $r(B(2,2))=8$.  Consider any $2$-coloring of $K_8$ and let $v \in V(K_8)$.  By pigeonhole principle, we may assume without loss of generality that $d_b(v) \geq 4$, but by Corollary~\ref{boundsondegrees} we know that $d_b(v) \leq 4$ as well, so $d_b(v) = 4$.  To avoid a blue $B(2,2)$ it must be that for every $u_b \in B_v$, $\{u_b,u_r\}$ is blue for at most one $u_r \in R_v$, and that for every $u_r \in R_v$, $\{u_b,u_r\}$ is red for at most one $u_b \in B_v$.  This is an obvious contradiction to the pigeonhole principle, when applied to the $12$ edges between the two sets $B_v$ and $R_v$.  Thus there must be a monochromatic copy of $B(2,2)$, and $r(B(2,2)) \leq 8$, as needed.

Now let $n \geq 3$, and consider any $2$-coloring of the edges of $K_{2n+3}$.  Fix $v$ and suppose without loss of generality that $d_b(v) \geq d_r(v)$.  Corollary~\ref{boundsondegrees} guarantees that $n \leq d_r(v) \leq d_b(v) \leq n+2$, and of course $d_b(v)+d_r(v) = d(v) = 2n+2$.

First suppose that $d_b(v)=n+2$ and $d_r(v)=n$.  As in the proof of the case $n=2$, in order to avoid a blue $B(2,n)$ every $u_b \in B_v$ shares a blue edge with at most one $u_r \in R_v$.  Furthermore, in order to avoid a red $B(2,n)$ every $u_r \in R_v$ shares a red edge with at most $n-1$ vertices $u_b \in B_v$.  Therefore there can be no more than $n(n-1)+n+2 = n^2+2$ edges between $B_v$ and $R_v$, which is clearly not so: the number of such edges is $n^2+2n > n^2+2$.  Therefore $K_{2n+3}$ must contain a monochromatic $B(2,n)$ in this case.

We obtain a similar contradiction in case $d_b(v)=d_r(v)=n+1$: now avoiding a monochromatic $B(2,n)$ forces there to be at most $2n+2$ edges between $B_v$ and $R_v$.  We therefore conclude that $r(B(2,n)) \leq 2n+3$ when $n \geq 3$.
\end{proof}

Note that since $B(2,n)$ contains $B(1,n)$ as a subgraph, we may actually conclude that \[ r(B(2,n)) \in \left\{ \begin{array}{rl} \{2n+1,2n+2,2n+3\} & n \ {\rm odd,} \\ \{2n+2,2n+3\} & n \ {\rm even}. \\ \end{array} \right. \]

Our next result establishes an upper bound on $r(B(m,n))$ for a large number of bistars with larger $m$, generalizing (2) of Proposition~\ref{smallm}.

\begin{proposition} \label{upperboundbign} Let $m,n \in \mathbb{N}$ such that $m \geq 3$ and $m+2 \leq n \leq 2m-1$. Then $r(B(m,n)) \leq 2n+m+1$. \end{proposition}

\begin{proof} Let us pick and fix $v_1 \in V$, and let $x \geq 0$ so that $R=2m+n+x+2$.  Following Corollary~\ref{boundsondegrees} we may let $d_b(v_1) = m+n-a$ and $d_r(v_1) = m+x+a+1$ for some $a \in [0,\lfloor \frac{n-x-1}{2} \rfloor]$.

First consider $v \in R_1$.  If there are at least $n$ vertices $u$ in $B_1$ such that $\{v,u\}$ is red, then $v_1$ and $v$ form the centers of a red $B(m,n)$.  Therefore we may assume that every vertex $v \in R_1$ shares a red edge with at most $n-1$ vertices in $B_1$.  The number of red edges between a vertex in $B_1$ and a vertex in $R_1$ is thus at most $(n-1)(m+x+a+1)$.

Note that \[ n+1 \leq m+n-a \Leftrightarrow a+1 \leq m. \] \noindent But $a+1 \leq \frac{n-x+1}{2}$, and \[ \frac{n-x+1}{2} \leq m \Leftrightarrow n \leq 2m+x-1. \] \noindent Since $n \leq 2m-1$ and $x \geq 0$, this last inequality holds, so that $n+1 \leq m+n-a = d_b(v_1)$.

Suppose now that for some $v \in B_1$, $v$ shares a blue edge with $a+1$ vertices in $R_1$.  Since $d_b(v) \geq m+1$, we know that $v$ shares a blue edge with at least $m-a$ vertices in $B_1$.  Pick $m-a$ such vertices, $u_1,...,u_{m-a}$, leaving $m+n-a-(m-a)=n$ vertices, $w_1,...,w_n$ in $B_1$.  We have now found a blue $B(m,n)$, with centers $v_1$ and $v$ and leaves $w_1,...,w_n,u_1,...,u_{m-a}$, the first set adjacent to $v_1$ and the second to $v$.

We may therefore assume that every $v \in B_1$ shares a blue edge with at most $a$ vertices in $R_1$.  As $v \in B_1$ is arbitrary there are at most $a(m+n-a)$ blue edges between a vertex in $B_1$ and a vertex in $R_1$.

The computations above show that the total number of edges shared by a vertex in $B_1$ and a vertex in $R_1$ is at most $(n-1)(m+x+a+1)+a(m+n-a)$.  However, the number of such edges is obviously $(m+x+a+1)(m+n-a)$.  Thus we obtain a contradiction when \[ (m+x+a+1)(m+n-a) > (n-1)(m+x+a+1)+a(m+n-a), \] \noindent or, after rearranging, \[ (m+1)(m+x+1) > a(m+n+x). \] \noindent This inequality follows from the stronger one obtained by replacing $a$ by the upper bound $\frac{n-x-1}{2}$: \[ (m+1)(m+x+1) > \left(\frac{n-x-1}{2}\right)(m+n+x). \] Finally, we rearrange this to obtain an inequality featuring a quadratic polynomial in $x$: \[ x^2+(3m+3)x+2(m+1)^2+(1-n)(m+n) > 0. \] \noindent The polynomial has
a unique positive root at $x=n-m-2$, so that our strict inequality holds when $x \geq n-m-1$.  This contradiction shows us that in such
cases we must be able to find a monochromatic $B(m,n)$.  Therefore we are guaranteed such a $B(m,n)$ whenever \[ R=2m+n+x+2 \geq 2m+n+n-m-1+2 = 2n+m+1. \] \noindent Thus $r(B(m,n)) \leq 2n+m+1$, as needed. \end{proof}

An entirely analogous proof can be used to verify the following result.

\begin{proposition} \label{upperboundsmalln} Suppose that $m \geq 2$ and $n \in \{m,m+1\}$.  Then $r(B(m,n)) \leq 2m+n+2$. \end{proposition}

Combining this with Proposition~\ref{propLowerbound}, we obtain

\begin{theorem} \label{equal} Suppose that $m \geq 2$ and $n \in \{m,m+1\}$.  Then $r(B(m,n))=2m+n+2$. \end{theorem}

\section{Ramsey saturation of bistars} \label{saturation}

\noindent It is not hard to show that stars $S_n$ are saturated.  Indeed, we showed in Section~\ref{secLowerbound} that $r(S_n) \leq 2n$ for any $n$; saturation of the star $S_n$ follows from our next result:

\begin{proposition} \label{starplusedge}
Let $G$ be the graph $S_n+e$, where $e \not\in E(S_n)$.  Then $r(G) \geq 2n+1$.  As a consequence, $S_n$ is Ramsey saturated.
\end{proposition}

\begin{proof}
Consider the graph $K_{2n}$ with vertex set $V = U \cup W \cup \{v\}$, $|U|=n$ and $|W|=n-1$.  Color each of the edges $\{u_i,u_j\}$ (for $u_i, u_j \in U$), $\{w_i,w_j\}$ and $\{v,w_i\}$ (for $w_i,w_j \in W$) blue, and color each of the edges $\{v,u_i\}$ and $\{u_i,w_j\}$ (for $u_i \in U$ and $w_j \in W$) red.  It is easy to check that this $2$-coloring of $K_{2n}$ yields no monochromatic copy of $S_n+e$, so that $r(S_n+e) \geq 2n+1$. \end{proof}

Schelp and his colleagues (see \cite{balisterlehelschelp} and \cite{schelp}) conjecture that all trees with at least $5$ vertices and which are not stars are unsaturated.  In this direction we offering the following

\begin{theorem} \label{bistarunsat} Let $m,n \in \mathbb{N}$. Then the bistar $B(m,n)$ is unsaturated. \end{theorem}

\begin{proof} Fix $m,n \in \mathbb{N}$ and let $r=r(B(m,n))$. (Recall that $r \geq 2m+n+2$, by Proposition~\ref{propLowerbound}.)  We show that every $2$-coloring of $K_r$ results in a monochromatic copy of the graph $B(m,n)+e$ where $e$ is an edge between leaves of $B(m,n)$ which are adjacent to different centers.  This shows that for such $e$, $r(B(m,n)+e)=r(B(m,n))$.

Suppose we are given a $2$-coloring of $K_r$.  By our choice of $r$, $k_r$ contains a monochromatic copy of $B(m,n)$.  Let us denote by $u$ and $v$ the centers of blue degree $m+1$ and $n+1$, respectively, with $B_u = \{u_1,...,u_m\}$, $B_v = \{v_1,...,v_n\}$.  Since $r \geq 2m+n+2$, the set $F = V \setminus \{u,v,u_1,...,u_m,v_1,...,v_n\}$ satisfies $|F| \geq m$.  Let $F=\{f_1,...,f_{\ell}\}$, $\ell \geq m$.

In order that there be no blue subgraph of the form $B(m,n)+e$ as described above, every edge $\{u_i,v_j\}$ must be red.  In order then to avoid the presence of a red $B(m,n)+e$ one of three conditions must hold:

\begin{enumerate}

\item all edges of the form $\{u_i,f_k\}$ or $\{u_i,v\}$ are blue,

\item all edges of the form $\{v_j,f_k\}$ or $\{v_j,u\}$ are blue, or

\item all edges of the form $\{u_i,v\}$ or $\{v_j,u\}$ are blue, and there exist indices $a$, $b$, and $c$ such that $\{u_a,f_c\}$ $\{v_b,f_c\}$ are red but all other vertices of the form $\{u_i,f_k\}$ or $\{v_j,f_k\}$ are blue.

\end{enumerate}

\medskip

\noindent \textbf{Case 1.} \ In this case, in order that the vertices $u_i$ and $v$ not be the centers of a blue subgraph of the form $B(m,n)+e$, it must be that all of the edges of the form $\{v_i,f_k\}$ and $\{u,f_k\}$ ($i$ and $k$ arbitrary) are red.  However, now any fixed pair $v_b$ and $f_c$ form the centers of a red $B(m,n)+e$, with red edges $\{f_c,u\}$ and $\{f_c,v_j\}$ ($j \neq b$), as well as $\{v_b,u_i\}$ ($i$ arbitrary), completing the bistar itself and any red edge $\{u_i,v_j\}$ ($j \neq b$) providing the needed edge $e$.

\medskip

\noindent \textbf{Case 2.} \ In this case, in order that the vertices $v_j$ and $f_k$ not be the centers of a blue $B(m,n)+e$, it must be that all of the edges of the form $\{v,f_j\}$ and $\{u_i,f_j\}$ ($i$ and $j$ arbitrary) are red.  However now any fixed pair $u_a$ and $f_c$ form the centers of a red $B(m,n)+e$, with red edges $\{f_c,u_i\}$ and $\{f_c,v\}$ ($i \neq a$), as well as $\{u_a,v_j\}$ ($j$ arbitrary), completing the bistar itself and any red edge $\{u_i,v_j\}$ ($i \neq a$) providing the needed edge $e$.

\medskip

\noindent \textbf{Case 3.} \ Let $a$, $b$, and $c$ be as above.  Note that the vertices $u$ and $v_b$ form the centers of a blue $B(m,n)+e$, with blue edges $\{u,u_a\}$ and $\{u,v_j\}$ ($j \neq b$), as well as $\{v_b,v\}$ and $\{v_b,f_k\}$ ($k \neq c$), completing the bistar itself and the blue edge $\{v,u_a\}$ providing the needed edge $e$.

\medskip

Therefore any $2$-coloring of $K_r$ yields a monochromatic copy of $B(m,n)+e$, so $r(B(m,n)+e) \leq r(B(m,n))$; the reverse inequality is obvious, so the two Ramsey numbers are equal. \end{proof}

\section{Future directions} \label{future}

Just as establishing a precise value for $r(C)$ for a more general caterpillar $C$ is difficult, proving saturation for such a $C$ is complicated by the fact that as the number of vertices in the spine $S(C)$ grows, the number of non-isomorphic graphs of the form $C+e$ larger very quickly; a proof of saturation would likely involve more sophisticated techniques from enumerative combinatorics in order to argue the unavoidability of a specific graph $C+e$.

It is conceivable that the techniques we have applied here may lead to an understanding for moderately more complicated trees, such as the bistars we have not discussed thoroughly, and perhaps trees with diameter at $4$. However, even in these latter trees the number of vertices of degree $3$ may grow arbitrarily large, and fundamentally different techniques may be needed to obtain reasonable estimates for $r(C)$ unless $C$ is highly ``regular'' in some way.


\begin{thebibliography}{99}

\bibitem{balisterlehelschelp} \textsc{Balister, P., Lehel, J.}, and \textsc{Schelp, R.H.}, ``Ramsey unsaturated and saturated graphs,'' \textit{J. Graph Theory} \textbf{51} (2006) no. 1, 22--32.

\bibitem{burr} \textsc{Burr, S.A.}, ``Generalized Ramsey theory for graphs,'' \textit{Graphs and combinatorics (Proc. Capital Conf., George Washington Univ., Washington, D.C., 1973)}, Lecture Notes in Mat., Vol. 406, Springer, Berlin (1974), 52--75.

\bibitem{burrroberts} \textsc{Burr, S.A.} and \textsc{Roberts, J.A.}, ``On Ramsey numbers for stars,'' \textit{Utilitas Math.} \textbf{4} (1973) 217--220.

\bibitem{erdosgraham} \textsc{Erd\H{o}s, P.} and \textsc{Graham, R.L.}, ``On partition theorems for finite graphs,'' \textit{Infinite and finite sets (Colloq., Keszthely, 1973; dedicated to P. Erdos on his 60th birthday), Vol. I}, Colloq. Math. Soc. Janos Bolyai, Vol. 10, North-Holland, Amsterdam (1975), 515--527.

\bibitem{grahamrothschildspencer} \textsc{Graham, R.L., Rothschild, B.L.}, and \textsc{Spencer, J.H.}, \textit{Ramsey theory} (2nd ed.), John Wiley \& Sons, New York (1990).

\bibitem{guovolkmann} \textsc{Guo, Y.} and \textsc{Volkmann, L.}, ``Tree-Ramsey numbers,'' \textit{Australas. J. Combin.} \textbf{11} (1995) 169--175.

\bibitem{schelp} \textsc{Schelp, R.H.}, ``Some Ramsey results and conjectures,'' \textit{22nd Cumberland Conference on Combinatorics, Graph Theory, and Computing} (May 2009), presentation.

\end{thebibliography}
\end{document}